\documentclass[
a4paper, 
12pt, 
BCOR=5mm,
fleqn, 
twoside, 
headsepline, 
parskip=half-,
]{scrartcl}

\usepackage[english]{babel}
\usepackage[utf8]{inputenc}
\usepackage[T1]{fontenc}
\usepackage[english,cleanlook]{isodate}
\usepackage{lmodern}
\usepackage{amsmath}
\usepackage{mathtools}
\usepackage{amssymb}
\usepackage{amsthm}
\usepackage{scrlayer-scrpage} 
\usepackage{enumitem}
\usepackage{tikz} 

\ohead{\footnotesize Mario B. Schulz}
\ihead{\footnotesize Yamabe flow on non-compact manifolds with unbounded initial curvature}
\setkomafont{pageheadfoot}{}

\providecommand{\Hyp}{\mathbb{H}}
\providecommand{\Sp}{\mathbb{S}}
\providecommand{\dd}{d}
\providecommand{\Rsc}{\mathrm{R}}
\providecommand{\ubelow}{c_1}
\providecommand{\uabove}{c_2}
\providecommand{\Rbelow}{\kappa_1}
\providecommand{\Rabove}{\kappa_2}

\DeclarePairedDelimiter\abs{\lvert}{\rvert}
\DeclarePairedDelimiter\nm{\lVert}{\rVert}
\DeclarePairedDelimiter\sk{\langle}{\rangle}
\DeclarePairedDelimiter\interval{]}{[}
\DeclarePairedDelimiter\Interval{[}{[}

\DeclareMathOperator{\Yamabe}{Y}

\theoremstyle{plain}
\newtheorem{thm}{Theorem} 
\newtheorem{lem}{Lemma}[section]

\theoremstyle{definition}
\newtheorem*{defn*}{Definition}
 
\theoremstyle{remark}
\newtheorem*{rem*}{Remark}
 
\usepackage[
pdftitle={Yamabe flow on non-compact manifolds with unbounded initial curvature},
pdfauthor={Mario B. Schulz}, 
pdfborder={0 0 0},
]{hyperref}

\title{Yamabe flow on non-compact manifolds with unbounded \\ initial curvature}
\author{Mario B. Schulz\thanks{This research was supported by the Swiss National Science Foundation under grant 200020\_159925.}
\\\footnotesize\itshape
ETH Z\"urich, R\"amistrasse 101, 8092 Z\"urich, Switzerland
}
\date{\today}
 
\begin{document}

\maketitle

\begin{abstract}
We prove global existence of Yamabe flows on non-compact manifolds $M$ of dimension $m\geq3$ under the assumption that the initial metric $g_0=u_0g_M$ is conformally equivalent to a complete background metric $g_M$ of bounded, non-positive scalar curvature and positive Yamabe invariant with conformal factor $u_0$ bounded from above and below.  
We do not require initial curvature bounds. 
In particular, the scalar curvature of $(M,g_0)$ can be unbounded from above and below without growth condition. 
\end{abstract}

Richard Hamilton's \cite{Hamilton1989} Yamabe flow describes a family of Riemannian metrics $g(t)$ subject to the evolution equation $\frac{\partial}{\partial t}g=-\Rsc_g\,g$, where $\Rsc_g$ denotes the scalar curvature corresponding to the metric $g$. 
This equation tends to conformally deform a given initial metric towards a metric of vanishing scalar curvature.  
Hamilton proved existence of Yamabe flows on compact manifolds without boundary. 
Their asymptotic behaviour was subsequently analysed by Chow \cite{Chow1992}, Ye \cite{Ye1994}, Schwetlick and Struwe \cite{Schwetlick2003} and Brendle \cite{Brendle2005, Brendle2007}. 
The theory of Yamabe flows on non-compact manifolds is not as developed as in the compact case. 
Daskalopoulos and Sesum \cite{Daskalopoulos2013} analysed the profiles of self-similar solutions (Yamabe solitons).  
The question of existence in general was addressed by 
Ma and An who obtained the following results on complete, non-compact Riemannian manifolds $(M,g_0)$ satisfying certain curvature assumptions: 
\begin{itemize}
\item If $(M,g_0)$ has Ricci curvature bounded from below and uniformly bounded, non-positive scalar curvature, then there exists a global Yamabe flow on $M$ with $g_0$ as initial metric \cite{Ma1999}. 
\pagebreak[1]
\item If $(M,g_0)$ is locally conformally flat with Ricci curvature bounded from below and uniformly bounded scalar curvature, then there exists a short-time solution to the Yamabe flow on $M$ with $g_0$ as initial metric \cite{Ma1999}. 
\item If $(M,g_0)$ has non-negative scalar curvature $\Rsc_{g_0}$ (possibly unbounded from above) and if the equation $-\Delta_{g_0} w=\Rsc_{g_0}$ has a non-negative solution $w$ in $M$, then there exists a global Yamabe flow on $M$ with $g_0$ as initial metric \cite{Ma2016}. 
\end{itemize}
More recently, Bahuaud and Vertman \cite{Bahuaud2014,Bahuaud2016} constructed Yamabe flows starting from spaces with incomplete edge singularities such that the singular structure is preserved along the flow. 
{Choi}, {Daskalopoulos}, and {King} \cite{Choi2018} were able to find solutions to the Yamabe flow on $\mathbb{R}^m$ which develop a type II singularity in finite time.

In dimension $m=2$, where the Yamabe flow coincides with the Ricci flow, Giesen and Topping \cite{Topping2010,Giesen2011,Topping2015} introduced the notion of instantaneous completeness and obtained existence and uniqueness of instantaneously complete Ricci flows on arbitrary surfaces. 
In particular, they do not require any assumptions on the curvature of the initial surface in order to prove existence of solutions. 
It is natural to ask whether Giesen and Topping's results generalise to non-compact manifolds of higher dimension. 

In \cite{Schulz2016}, the author obtained existence of instantaneously complete Yamabe flows on hyperbolic space of arbitrary dimension $m\geq3$ provided the initial metric is conformally hyperbolic with conformal factor and scalar curvature bounded from above.  
The goal of this paper is to construct complete Yamabe flows on non-compact manifolds $M$ of dimension $m\geq3$ starting from some complete initial metric $g_0$ but \emph{without} any curvature assumption on $(M,g_0)$. 
In particular, the initial scalar curvature $\Rsc_{g_0}\colon M\to\mathbb{R}$ is allowed to be unbounded from above and below.   
Instead we assume $g_0$ to be conformally equivalent to some ``well-behaved'' background metric $g_{M}$ on $M$ and only require pointwise bounds on the conformal factor $u_0$ characterising $g_0=u_0 g_M$. 
More precisely, we prove the following statement. 

\begin{thm}\label{thm:existence}
Let $(M,g_M)$ be a complete, non-compact Riemannian manifold of dimension $m\geq3$ with positive Yamabe invariant and non-positive, bounded scalar curvature $-\Rabove\leq\Rsc_{g_{M}}\leq-\Rbelow\leq0$. 
Let $g_0=u_0g_{M}$ be any conformal metric on $M$ with conformal factor $u_0\in C^{2,\alpha}(M)$ for some $0<\alpha<1$ satisfying 
\[0<\ubelow\leq u_0\leq \uabove<\infty.\]
Then, there exists a global Yamabe flow $(g(t))_{t\in\Interval{0,\infty}}$ on $M$ satisfying 
\begin{enumerate}[label={\normalfont(\arabic*)}]
\item $g(0)=g_0$ in the sense that $g(t)$ converges locally smoothly to $g_0$ as $t\searrow 0$,
\item $(\Rbelow t+\ubelow)\,g_{M}\leq g(t)\leq (\Rabove t+\uabove)\,g_{M}$ for every $t>0$,
\item $\Rsc_{g(t)}\geq-\frac{1}{t}$ for every $t>0$. 
\end{enumerate}
\end{thm}

\begin{rem*}
Typical examples of background manifolds $(M,g_M)$ satisfying the hypothesis of Theorem~\ref{thm:existence} are Euclidean space $(\mathbb{R}^m,g_{\mathbb{R}^m})$ with $\Rsc_{g_{\mathbb{R}^m}}\equiv0$ and hyperbolic space $(\Hyp,g_{\Hyp})$ with $\Rsc_{g_{\Hyp}}\equiv-m(m-1)$. 
We verify positivity of their Yamabe invariant in Lemma \ref{lem:Yamabeinvariant}. 
\end{rem*}

\section{Local existence}\label{sec:1}

Let $g_0=u_0g_M$ be any conformal metric on $(M,g_M)$. 
Since the Yamabe flow preserves the conformal class of the initial metric, any Yamabe flow $(g(t))_{t\in[0,T]}$ on $M$ with $g(0)=g_0$ is given by $g(t)=u(\cdot,t)\,g_M$ with some function $u\colon M\times[0,T]\to\interval{0,\infty}$. 
Let 
\begin{align*}
\eta\vcentcolon=\frac{m-2}{4}
\end{align*}
where $m=\dim M\geq3$ and $U=u^{\eta}$. 
Then, the metric $g=u g_M$ has scalar curvature 
\begin{align*}
\Rsc_g=U^{-\frac{m+2}{m-2}}\bigl(\Rsc_{g_{M}}U-4\tfrac{m-1}{m-2}\Delta_{g_{M}}U\bigr). 
\end{align*}
Hence, the conformal factor $u$ is subject to the evolution equation  
\begin{align*}
\frac{\partial u}{\partial t}
=-u\Rsc_g
&=-u^{1-\frac{m+2}{4}}\bigl(\Rsc_{g_{M}}u^{\eta}-4\tfrac{m-1}{m-2}\Delta_{g_{M}}u^{\eta}\bigr)
\\\notag
&= -\Rsc_{g_M}+ \tfrac{m-1}{\eta}u^{-\eta}\Delta_{g_{M}}u^{\eta} 
\\\notag
&=-\Rsc_{g_M}+(m-1)\bigl(u^{-1}\Delta_{g_M} u+(\eta-1)u^{-2}\abs{\nabla u}_{g_M}^2\bigr)
\end{align*}
which can also be expressed in the form 
\begin{align}\label{eqn:180609}
\frac{1}{\eta+1}\frac{\partial U^{1+\frac{1}{\eta}}}{\partial t}
&=-\Rsc_{g_M}U+\frac{m-1}{\eta}\Delta_{g_{M}}U.
\end{align}
Let $B_r\subset M$ denote the open metric ball of radius $r$ around some origin $p_0\in M$ and let $M_1\subset M_2\subset\ldots\subset M$ be an exhaustion of $M$ with smooth, open, connected, bounded domains such that $B_k\subset M_k\subset B_{k+1}$ for every $k\in\mathbb{N}$.
Fix any radius $4<k\in\mathbb{N}$. 
Let $\varphi\colon M\to[0,1]$ be smooth with compact support in $B_k$ such that $\varphi|_{B_{k-1}}\equiv1$. 
Under the assumption $u_0\in C^{2,\alpha}(\overline{M_k})$ for some $0<\alpha<1$ and $\ubelow\vcentcolon=\inf_{M}u_0>0$ we consider 
\begin{align}\label{eqn:initialdata}
\mathring{u}_0\vcentcolon=(1-\varphi)\ubelow+\varphi\, u_0 
\end{align}
as initial data for the Yamabe flow equation on $M_k$. 
\begin{figure}[htb]
\centering
\begin{tikzpicture}[line cap=round,line join=round,yscale=0.5,scale=1.25,
declare function={phi(\x)=exp(-1/(1-\x))/(exp(-1/(1-\x))+exp(-1/(\x)));},
declare function={fct(\x)=3.8+1.2*cos(deg(\x))-1.5*sin(deg(1.6*\x));},
]
\pgfmathsetmacro{\r}{5}
\pgfmathsetmacro{\c}{1.1}
\pgfmathsetmacro{\samp}{100}
\draw[->](-\r-0.5,0)--(\r+0.5,0)node[right]{$d$};
\draw[->](0,0)--(0,6.5);
\draw(-\r,0)node{$\scriptstyle+$};
\draw(0,0)node{$\scriptstyle+$};
\draw(1-\r,0)node{$\scriptstyle+$};
\draw(0,\c)node{$\scriptstyle+$}node[above right]{$\ubelow$};
\draw(\r,0)node{$\scriptstyle+$}node[below]{$r\vphantom{(-1)}$};
\draw(\r-1,0)node{$\scriptstyle+$}node[below]{$(r-1)$};
\draw[dotted](-\r-0.5,\c)--(\r+0.5,\c);
\draw[dotted](\r-1,0)--++(0,{fct(\r-1)});
\draw[dotted](1-\r,0)--++(0,{fct(1-\r)});
\draw[semithick,domain=-\r+1:\r-1, samples={\samp},smooth, variable=\x]plot ({\x},{fct(\x)});
\draw[semithick,dashed,domain=0.001:1,samples={11*\samp/100},smooth, variable=\x]plot (\x-\r,{fct(\x-\r)})plot (\x+\r-1,{fct(\x+\r-1)});
\draw[semithick,domain=0.0001:1,samples={\samp/2},smooth, variable=\x]plot (\x-\r,{(1-phi(1-\x))*\c+phi(1-\x)*fct(\x-\r)});
\draw[semithick,domain=0.0001:1,samples={\samp/2},smooth, variable=\x]plot (\x+\r-1,{(1-phi(\x))*\c+phi(\x)*fct(\x+\r-1)});
\draw(-\r,{fct(-\r)})node[above]{$u_0$};
\draw(-\r,\c)node[above]{$\mathring{u}_0$};
\end{tikzpicture}
\label{fig:initialdata}
\caption{Initial data $\mathring{u}_0$ for problem \eqref{eqn:pde}.}
\end{figure}
In particular, $0<\mathring{u}_0\in C^{2,\alpha}(M_k)$ coincides with $u_0$ in $B_{k-1}$ and takes the constant value $\ubelow$ in some neighbourhood of $\partial M_k$ as illustrated in Figure \ref{fig:initialdata}. 
As parabolic boundary data, we choose 
\begin{align}\label{eqn:boundarydata}
\phi(x,t)\vcentcolon=\ubelow-t\,\Rsc_{g_M}(x). 
\end{align}
Then, since $\mathring{u}_0$ and $\phi$ satisfy the first-order compatibility
conditions, there exists some $T>0$ (which a priori depends on $k$) and a solution $0<u\in C^{2,\alpha;1,\frac{\alpha}{2}}(\overline{M_k}\times[0,T])$ of 
\begin{align}\label{eqn:pde}
\left\{
\begin{aligned}
\frac{1}{m-1}\frac{\partial u}{\partial t}&=\frac{-\Rsc_{g_M}}{m-1}+\frac{\Delta_{g_M} u}{u}+\frac{(m-6)}{4}\frac{\abs{\nabla u}_{g_M}^2}{u^2} 
&&\text{ in }M_k\times[0,T],
\\
u&=\phi 
&&\text{ on }\partial M_k\times[0,T],
\\[1ex]
u&=\mathring{u}_0
&&\text{ on } M_k\times\{0\}.
\end{aligned}\right.
\end{align}
In \cite[Lemma 1.1]{Schulz2016}, the author gives a detailed proof of this short-time existence result for bounded domains in hyperbolic space using the inverse function theorem on Banach spaces. This approach generalises to smooth, bounded domains $M_k$ in any Riemannian manifold. 

\begin{lem}[Upper and lower bound]\label{lem:pointwise}
Let $0<u\in C^{2;1}(M_k\times[0,T])$ be a solution to problem \eqref{eqn:pde} with boundary data \eqref{eqn:boundarydata} and initial data \eqref{eqn:initialdata}. 
If the background scalar curvature satisfies $0\leq\Rbelow\leq-\Rsc_{g_M}\leq\Rabove$ in $M_k$, then
\begin{align*}
\ubelow+\Rbelow t\leq u(\cdot,t)\leq \uabove+\Rabove t
\end{align*}
for every $0\leq t\leq T$, where we recall $\ubelow=\inf u_0$ and $\uabove=\sup u_0$. 
\end{lem}

\begin{proof}
Equation \eqref{eqn:pde} implies that the function $w(\cdot,t)=u(\cdot,t)-\ubelow-\Rbelow t$ satisfies
\begin{align}\label{eqn:180510-1}
\frac{1}{m-1}\frac{\partial w}{\partial t}
-\frac{\Delta_{g_M} w}{u}-\frac{(m-6)}{4u^2}\sk{\nabla u,\nabla w}_{g_M}
&=\frac{-\Rsc_{g_M}-\Rbelow}{m-1}\geq0. 
\end{align}
Since $u>0$, equation \eqref{eqn:180510-1} is uniformly parabolic. 
Moreover, we have $w\geq0$ on $(\partial M_k\times[0,T])\cup(M_k\times\{0\})$. 
Hence, the linear parabolic maximum principle (see \cite[Prop. A.2]{Schulz2016}) implies $w\geq0$ in $M_k\times[0,T]$. 
The proof of the upper bound is analogous. 
\end{proof}

\begin{lem}[global existence on bounded domains]
\label{lem:globalexistence}
If the background scalar curvature satisfies $0\leq\Rbelow\leq-\Rsc_{g_M}\leq\Rabove$ in $M_k$, then there exists a unique solution $0<u\in C^{2;1}(M_k\times\Interval{0,\infty})$ to problem \eqref{eqn:pde} with boundary data \eqref{eqn:boundarydata} and initial data \eqref{eqn:initialdata}. 
\end{lem}

\begin{proof}
We invoke the same argument as in \cite{Schulz2016}. 
First we show that two positive solutions $u,v\in C^{2,\alpha;1,\frac{\alpha}{2}}(\overline{M_k}\times[0,T])$ of problem~\eqref{eqn:pde} with equal initial and boundary data must agree. We have
\begin{align*}
\frac{1}{m-1}\frac{\partial}{\partial t}(u-v)
&=\frac{\Delta_{g_M} u}{u}-\frac{\Delta_{g_M} v}{v}+\frac{m-6}{4}\Bigl(\frac{\abs{\nabla u}_{g_M}^2}{u^2}-\frac{\abs{\nabla v}_{g_M}^2}{v^2}\Bigr)
\\[1ex]
&=\frac{\Delta_{g_M} (u-v)}{u}+\frac{m-6}{4u^2}\sk[\big]{\nabla(u+v),\nabla(u-v)}_{g_M} 
\\&\hphantom{{}={}}-\frac{\Delta_{g_M} v}{uv}(u-v) -\frac{m-6}{4}\frac{\abs{\nabla v}_{g_M}^2}{u^2v^2}(u+v)(u-v)
\end{align*}
which can be considered as linear parabolic equation for $u-v$ with bounded coefficients because $u,v\in C^{2,\alpha;1,\frac{\alpha}{2}}(\overline{M_k}\times[0,T])$ are uniformly bounded away from zero and from above by Lemma \ref{lem:pointwise} and because $\abs{\nabla u}$, $\abs{\nabla v}$, $\Delta v$ are bounded functions in $\overline{M_k}$.   
Since $(u-v)$ vanishes along $(M_k\times\{0\})\cup(\partial M_k\times[0,T])$, the parabolic maximum principle (see \cite[Prop A.2]{Schulz2016}) implies $u-v=0$ in $\overline{M_k}\times[0,T]$ as claimed.

It remains to show that the solution can be extended globally in time. 
Let $T_*>0$ be the supremum over all $T>0$ such that problem \eqref{eqn:pde} has a solution defined in $M_k\times[0,T]$. 
As shown above, two such solutions agree on their common domain. 
Therefore, there exists a solution $u$ defined on $M_k\times\Interval{0,T_*}$. 
Suppose $T_*<\infty$. 
The function $U=u^{\eta}$ for $\eta=\frac{m-2}{4}$ satisfies equation \eqref{eqn:180609} which can be written in divergence form
\begin{align}
\frac{1}{m-1}\frac{\partial u^{\eta+1}}{\partial t}
\label{eqn:divergenceform}
&=-\frac{(\eta+1)\Rsc_{g_M}}{(m-1)u}u^{\eta+1}+\operatorname{div}_{g_{M}}\Bigl(\frac{1}{u}\nabla u^{\eta+1}\Bigr).  
\end{align}
Since $\abs{\Rsc_{g_M}}\leq\Rabove$ and since $0<\ubelow\leq u\leq \uabove+\Rabove T_*$ by Lemma \ref{lem:pointwise}, equation \eqref{eqn:divergenceform} can be interpreted as linear parabolic equation with uniformly bounded coefficients.  
Hence, parabolic DeGiorgi--Nash--Moser theory \cite[§\,4]{Trudinger1968} applies and yields
\begin{align*}
\nm{u^{\eta+1}}_{C^{0,\alpha;0,\frac{\alpha}{2}}(\overline{M_k}\times[0,T_*])} 
\leq C(m,\ubelow,\uabove,\Rabove,T_{*}) 
\end{align*}
for some Hölder exponent $0<\alpha<1$ and some constant $C$ depending only on the indicated constants. 
Together with Lemma \ref{lem:pointwise} we conclude, that $\frac{1}{u}$ is Hölder continuous
and apply linear parabolic theory \cite[\textsection\,IV.5, Theorem 5.2]{Ladyzenskaja1967} to obtain 
\begin{align*}
\nm{u}_{C^{2,\alpha;1,\frac{\alpha}{2}}(\overline{M_k}\times[0,T_*])} 
\leq C'(m,\ubelow,\uabove,\Rabove,T_{*}).
\end{align*} 
Hence, $u(\cdot,T_*)\in C^{2,\alpha}(\overline{M_k})$ are suitable initial data for problem \eqref{eqn:pde}. The boundary data \eqref{eqn:boundarydata} are defined also for $t\geq T_*$ and are compatible with $u(\cdot,T_*)$ at time $T_*$ Therefore, we can extend the solution in contradiction to the maximality of $T_*$. 
\end{proof}

\section{Scalar curvature estimates}
We assume the background scalar curvature to satisfy $0\leq\Rbelow\leq-\Rsc_{g_M}\leq\Rabove$ in $M$. 
From section \ref{sec:1} we recall the exhaustion of $M$ with smooth, connected, bounded domains $M_1\subset M_2\subset\ldots\subset M$ such that $B_k\subset M_k\subset B_{k+1}$ for every $k\in\mathbb{N}$. 
Let $T>0$ be arbitrary and $k\in\mathbb{N}$. 
By Lemma \ref{lem:globalexistence} there exists a unique solution $u\colon M_k\times[0,T]\to\mathbb{R}$  of problem \eqref{eqn:pde}. 
The goal of this section is to estimate the corresponding scalar curvature $\Rsc_{ug_M}$ independently of the index $k$. 
 
\begin{lem}\label{lem:Rsc} 
Let $u\in C^{2,\alpha;1,\frac{\alpha}{2}}(M_k\times[0,T])$ be a solution to problem~\eqref{eqn:pde}. 
For every $t\in[0,T]$, the scalar curvature $\Rsc_{g(t)}$ of the Riemannian metric $g(t)=u(\cdot,t)g_{M}$ satisfies  
\begin{align*}
\Rsc_{g(t)}\geq-\frac{1}{t} \quad\text{ in }M_k.
\end{align*}
\end{lem}

\begin{proof}
Let $w(t)=\frac{1}{\varepsilon+t}$, where $0<\varepsilon<\frac{\ubelow}{\Rabove}$ is chosen such that $\Rsc_{g(0)}>-w(0)$ in $M_k$.
In $M_k\times[0,T]$ we can express the scalar curvature in the form $\Rsc_g=-\frac{1}{u}\frac{\partial u}{\partial t}$. 
In particular, the choice of boundary data \eqref{eqn:boundarydata} implies 
\begin{align*} 
\Rsc_g \vert_{\partial M_k\times[0,T]}&=-\frac{1}{\phi}\frac{\partial\phi}{\partial t}
=-\frac{-\Rsc_{g_M}}{\ubelow-t\,\Rsc_{g_M}}
\geq-\frac{1}{\frac{\ubelow}{\Rabove}+t}
\geq-w(t)
\quad\text{ on $\partial M_k\times[0,T]$.}
\end{align*}
Scalar curvature evolves by (see \cite[Lemma 2.2]{Chow1992})
\begin{align}
\label{eqn:evoR}
\tfrac{\partial}{\partial t}\Rsc_g=(m-1)\Delta_{g}\Rsc_g+\Rsc_g^2
\quad\text{ in $M_k\times[0,T]$.}
\end{align}
Combined with $\frac{\dd w}{\dd t}=-w^2$, we obtain
\begin{align*}
\tfrac{\partial}{\partial t}(\Rsc_g+w)
-(m-1)\Delta_{g}(\Rsc_g+w)-(\Rsc_g-w)(\Rsc_g+w)=0.
\end{align*}
Since $(\Rsc_g-w)\leq\Rsc_g$ is bounded from above in $M_k\times[0,T]$ and since the operator $\Delta_g=\frac{1}{u}\Delta_{g_M}+\frac{m-2}{2}u^{-2}\sk{\nabla u,\nabla\cdot}_{g_M}$  is uniformly elliptic, the inequality $\Rsc_g\geq -w$ in $M_k\times[0,T]$ follows from the parabolic maximum principle (see \cite[Prop. A.2]{Schulz2016}).  
\end{proof}

Integral estimates for the positive part of the scalar curvature can be obtained with the help of Sobolev-type inequalities. 
This technique requires the Yamabe invariant of $(M,g_M)$ to be positive. 

\begin{lem}[Yamabe invariant]\label{lem:Yamabeinvariant}
Let $(M,g)$ be a Riemannian manifold with scalar curvature $\Rsc_g$. 
Then the Yamabe invariant $\Yamabe$ of $(M,g)$ is defined by 
\begin{align*}
\Yamabe(M ,g)&\vcentcolon=\inf\left\{\frac{\displaystyle\int_M \abs{\nabla^{g}f}_g^2\,\dd\mu_g
+\frac{m-2}{4(m-1)}\int_M \Rsc_{g}f^2\,\dd\mu_g}{\displaystyle\Bigl(\int_{M }\abs{f}^{\frac{2m}{m-2}}\,\dd\mu_g\Bigr)^{\frac{m-2}{m}}}\,;~f\in C^{\infty}_{c}(M )\right\}.
\end{align*}
The Yamabe invariants of hyperbolic space $(\Hyp^m,g_{\Hyp^m})$, Euclidean space $(\mathbb{R}^m,g_{\mathbb{R}^m})$ and the round sphere $(\Sp^m,g_{\Sp^{m}})$ coincide. Their value is
\begin{align}\label{eqn:Yamabe-invariant}
\Yamabe(\Hyp^m,g_{\Hyp^m})
=\Yamabe(\mathbb{R}^m,g_{\mathbb{R}^m})
=\Yamabe(\Sp^m,g_{\Sp^{m}})
=\frac{m(m-2)}{4}\abs{\Sp^m}^{\frac{2}{m}}>0.
\end{align}
\end{lem}
\begin{proof}
The Yamabe invariant is a conformal invariant. 
Hyperbolic space is conformally equivalent to the Euclidean ball $B_r$ of any given radius $r>0$. 
Via stereographic projection, the sphere minus a point is conformally equivalent to $(\mathbb{R}^m,g_{\mathbb{R}^m})$. 
The $H^1$-capacity of a point vanishes. 
Hence, a minimising sequence for $\Yamabe(\Hyp^m,g_{\Hyp^m})$ yields  competitors for $\Yamabe(\Sp^m,g_{\Sp^m})$ and vice versa. 
Since $\Yamabe(\Sp^m,g_{\Sp^m})$ is attained for constant functions $w$ and since $\Rsc_{g_{\Sp^m}}=m(m-1)$, the claim follows. 
\end{proof}
 
The definition of Yamabe invariant leads to a Sobolev-type inequality. 
Let $g=u g_{M}$ be any conformal metric on $(M,g_{M})$. 
Let $B\subset M$ be open and $f\in C_c^{1}(B)$. 
Computing 
\begin{align}\label{eqn:180614}
\abs{\nabla^g f}_g^2=\frac{1}{u}\abs{\nabla^{g_M} f}_{g_M}^2
\end{align}
and denoting $\Yamabe\vcentcolon=\Yamabe(M,g_M)$, we obtain 
\begin{align}\notag
&\int_{B}\abs{\nabla^g f}_g^2\,\dd\mu_g
=\int_{B}\abs{\nabla^{g_{M }} f}_{g_{M }}^2u^{\frac{m}{2}-1}\,\dd\mu_{g_{M }}
\\[1ex]
&\geq\inf_{B} u^{\frac{m}{2}-1}\int_{M }\abs{\nabla^{g_{M }}f}_{g_{M }}^2 \,\dd\mu_{g_{M }}
\notag\\[1ex]\label{eqn:180611-2}
&\geq\inf_{B} u^{\frac{m}{2}-1}\biggl(\Yamabe{} 
\Bigl(\int_{M }\abs{f}^{\frac{2m}{m-2}}\,\dd\mu_{g_{M }}\Bigr)^{\frac{m-2}{m}}
-\frac{(m-2)}{4(m-1)}\int_{M }\Rsc_{g_{M }}f^2\,\dd\mu_{g_{M }}\biggr).
\end{align}
As in Lemma \ref{lem:pointwise} we will assume henceforth that the background manifold $(M,g_M)$ has scalar curvature $\Rsc_{g_{M}}$ satisfying
\begin{align}\label{eqn:backgroundcurvature}
0\leq\Rbelow\leq-\Rsc_{g_M}\leq \Rabove<\infty.
\end{align}
Lemma \ref{lem:pointwise} is the main reason for assumption \eqref{eqn:backgroundcurvature} but it also allows us to drop the second term in \eqref{eqn:180611-2} because of its sign to obtain the Sobolev inequality 
\begin{align}\label{eqn:Sobolev}
\int_{B}\abs{\nabla^g f}_g^2\,\dd\mu_g 
&\geq \Bigl(\frac{\inf_{B} u}{\sup_{B} u}\Bigr)^{\frac{m-2}{2}}\Yamabe
\Bigl(\int_{B}\abs{f}^{\frac{2m}{m-2}}\,\dd\mu_{g}\Bigr)^{\frac{m-2}{m}}.
\end{align}
 
\begin{lem}\label{lem:180507-1}
Given $k>4$, let $u\in C^{2,\alpha;1,\frac{\alpha}{2}}(M_k\times[0,T])$ be a solution to problem~\eqref{eqn:pde} and let $\Rsc_{g(t)}$ be the scalar curvature of the Riemannian metric $g(t)=u(\cdot,t)g_{M}$ in $M_k$. 
Let the radius $0<r_0<k-4$ be fixed and let $p>1$ be any exponent. 
Then, for every $t\in[0,T]$, the positive part $\Rsc_+(x,t)=\max\{0,\Rsc_{g(t)}(x)\}$ satisfies 
\begin{align*}
\int_{B_{r_0}}\Rsc_+^p(\cdot,t)\,\dd\mu_{g(t)}\leq C
\end{align*}
where the constant $C$ increases with $p$ and $r_0$ but does not depend on $k$. 
\end{lem}

\begin{proof}
For any exponent $p>1$ and any $\psi\in C_c^{\infty}(M_k)$, we have 
\begin{align}\notag
&\frac{\partial}{\partial t}\int_{M_k}\psi\Rsc_+^p\,\dd\mu_g
\\[1ex]\notag
&=p\int_{M_k}\psi\Rsc_+^{p-1}\bigl(\Rsc^2+(m-1)\Delta_g\Rsc\bigr)\,\dd\mu_g
-\frac{m}{2}\int_{M_k}\psi\Rsc_+^p\Rsc\,\dd\mu_g
\\[1ex]\notag
&=(p-\tfrac{m}{2})\int_{M_k}\psi\Rsc_+^{p+1}\,\dd\mu_g
+p(m-1)\int_{M_k}\psi\Rsc_+^{p-1}\Delta_g\Rsc\,\dd\mu_g
\\[1ex]\label{eqn:180502-1}
&=(p-\tfrac{m}{2})\int_{M_k}\psi\Rsc_+^{p+1}\,\dd\mu_g
\\&\hphantom{{}={}}\label{eqn:180502-2}
+(m-1)\Bigl(-\int_{M_k}\sk{\nabla\psi,\nabla\Rsc_+^p}_g\,\dd\mu_g
-\frac{4(p-1)}{p}\int_{M_k}\psi\abs{\nabla\Rsc_+^{\frac{p}{2}}}_g^2\,\dd\mu_g\Bigr).
\end{align}
The strategy of the proof is as follows. 
\begin{enumerate}[label={\emph{Step \arabic*.}},wide]
\item If $1<p<\frac{m}{2}$ then the negative sign of \eqref{eqn:180502-1} leads to a bound uniformly in $k$. 
\item We use the estimate obtained in the first step to deal with the case $p=\frac{m}{2}$. 
\item The bound from the second step can be extended to $p=\beta\frac{m}{2}$ for some $\beta>1$. 
\item Using the estimate for $p=\beta\frac{m}{2}$ we obtain a bound with any exponent $p>\frac{m}{2}$. 
\end{enumerate}
In each step we choose a different cut-off function $\psi$ and apply different estimates to control the terms in \eqref{eqn:180502-1} and \eqref{eqn:180502-2}. 
Young's inequality appears frequently in either of the following forms. 
\begin{alignat}{3}
\forall a,b&\geq0 \quad &&\forall 0<s<1: ~ &
a^s b^{1-s}&\leq s a+(1-s)b, 
\label{eqn:Young1}
\\[1ex]
\forall b,c&\geq0 \quad  &&\forall x\in\mathbb{R}: &
bx-c x^2&\leq\tfrac{b^2}{4c}.
\label{eqn:Young2}
\end{alignat}

\emph{Step 1.}  
Let $\varphi_1\in C_c^{\infty}(B_{r_0+4})$ be a cut-off function such that $0\leq\varphi_1\leq1$, such that $\varphi_1|_{B_{r_0+3}}\equiv1$ and such that $\abs{\nabla \varphi_1}_{g_{M}}\leq 2$. 
Given $1<p<\frac{m}{2}$, we set 
\begin{align*}
\psi_1=\varphi_1^{2(p+1)}.
\end{align*}
This cut-off function has the property that
\begin{align*}
\abs{\nabla\psi_1}_{g_{M }}^2
&\leq\bigr(4(p+1)\bigr)^2\varphi_1^{4p+2}
=\bigr(4(p+1)\bigr)^2\psi_1^{1+\frac{p}{p+1}}.
\end{align*} 
Choosing $\psi=\psi_1$ and recalling 
\begin{align*}
\abs{\nabla\psi_1}_g^2=\frac{1}{u}\abs{\nabla\psi_1}_{g_M}^2
\end{align*}
the terms in \eqref{eqn:180502-2} can be estimated as follows using Young's inequality and H\"older's inequality:  
\begin{align}\notag
&-\int_{M_k}\sk{\nabla\psi_1,\nabla\Rsc_+^p}_g\,\dd\mu_g
-\frac{4(p-1)}{p}\int_{M_k}\psi_1\abs{\nabla\Rsc_+^{\frac{p}{2}}}_g^2\,\dd\mu_g
\\[1ex]\notag
&\leq\frac{p}{4(p-1)}\int_{M_k}\frac{\abs{\nabla\psi_1}_g^2}{\psi_1}\,\Rsc_+^p\,\dd\mu_g
\\[1ex]\notag
&\leq\frac{4p(p+1)^2}{(p-1)}\int_{B_{r_0+4}}\frac{\psi_1^{\frac{p}{p+1}}}{u}\,\Rsc_+^p \,\dd\mu_g
\\[1ex]\notag
&\leq\frac{4p(p+1)^2}{(p-1)}\Bigl(\int_{B_{r_0+4}}u^{\frac{m}{2}-p-1}\,\dd\mu_{g_{M }}\Bigr)^{\frac{1}{p+1}}
\Bigl(\int_{M_k}\psi_1\Rsc_+^{p+1}\,\dd\mu_g\Bigr)^{\frac{p}{p+1}}
\\[1ex]\label{eqn:180611-1}
&\leq\frac{{(4p(p+1)^2)}^{p+1}}{(p+1)\lambda^p(p-1)^{p+1}}
\Bigl(\int_{B_{r_0+4}}u^{\frac{m}{2}-p-1}\,\dd\mu_{g_{M }}\Bigr)
+\frac{\lambda p}{p+1}\Bigl(\int_{M_k}\psi_1\Rsc_+^{p+1}\,\dd\mu_g\Bigr)
\end{align}
where the parameter $\lambda>0$ is arbitrary. 
If we choose $1<p<\frac{m}{2}$ and $\lambda=\frac{(\frac{m}{2}-p)(p+1)}{(m-1)p}$, then we obtain
\begin{align}\label{eqn:180430}
\frac{\partial}{\partial t}\int_{M_k}\psi_1\Rsc_+^p\,\dd\mu_g
&\leq C_{m,p}\Bigl(\int_{B_{r_0+4}}u^{\frac{m}{2}-p-1}\,\dd\mu_{g_{M }}\Bigr). 
\end{align}
Moreover, since $u(\cdot,t)\geq \ubelow+\Rbelow t$ by Lemma \ref{lem:pointwise}, the right hand side of \eqref{eqn:180430} is integrable in $t\in[0,T]$ if $1<p<\frac{m}{2}$. 
In particular, for $p=\frac{2m}{5}\in\interval{1,\frac{m}{2}}$ we obtain 
\begin{align}\label{eqn:180502-3}
\int_{B_{r_0+3}}\Rsc_{+}^{\frac{2m}{5}}\,\dd\mu_g
&\leq\int_{B_{r_0+4}}\abs{\Rsc_{g(0)}}^{\frac{2m}{5}}\,\dd\mu_{g(0)}
+C_{m,r_0,T}. 
\end{align}

\emph{Step 2.} For any regular, non-negative functions $\psi,F$ and any exponent $p$, there holds
\begin{align}
\abs[\big]{\nabla(\psi F^p)^{\frac{1}{2}}}^2
&=\abs[\big]{\tfrac{1}{2}\psi^{-\frac{1}{2}}F^{\frac{p}{2}}\nabla\psi+\psi^{\frac{1}{2}}\nabla F^{\frac{p}{2}}}^2
\notag\\
&=\frac{ F^p}{4\psi}\abs{\nabla\psi}^2
+ F^{\frac{p}{2}}\sk{\nabla\psi,\nabla F^{\frac{p}{2}}}
+\psi\abs{\nabla F^{\frac{p}{2}}}^2.  
\notag
\shortintertext{Therefore, }
-\frac{1}{2}\sk{\nabla\psi,\nabla\Rsc_+^{p}}_g
&=\frac{\abs{\nabla\psi}_g^2}{4\psi}\Rsc_+^p
-\abs[\big]{\nabla(\psi\Rsc_+^p)^{\frac{1}{2}}}_g^2
+\psi\abs{\nabla\Rsc_+^{\frac{p}{2}}}_g^2,  
\label{eqn:skpsiRsc}
\end{align}
where $\psi$ is a new cut-off function to be chosen. 
Given $p\geq\frac{m}{2}\geq\frac{3}{2}$, we use \eqref{eqn:skpsiRsc} to estimate the terms in \eqref{eqn:180502-2} by 
\begin{align*} 
&-\int_{M_k}\sk{\nabla\psi,\nabla\Rsc_+^p}_g\,\dd\mu_g
-\frac{4(p-1)}{p}\int_{M_k}\psi\abs{\nabla\Rsc_+^{\frac{p}{2}}}_g^2\,\dd\mu_g
\\[1ex] 
&=-\frac{1}{2}\int_{M_k}\sk{\nabla\psi,\nabla\Rsc_+^p}_g\,\dd\mu_g
+\Bigl(1-\frac{4(p-1)}{p}\Bigr)\int_{M_k}\psi\abs{\nabla\Rsc_+^{\frac{p}{2}}}_g^2\,\dd\mu_g
\\&\hphantom{{}={}} 
-\int_{M_k}\abs[\big]{\nabla(\psi\Rsc_+^p)^{\frac{1}{2}}}_g^2\,\dd\mu_g+\frac{1}{4}\int_{M_k}\frac{\abs{\nabla\psi}^2_g}{\psi}\Rsc_+^p\,\dd\mu_g
\\[1ex]
&\leq-\int_{M_k}\abs[\big]{\nabla(\psi\Rsc_+^p)^{\frac{1}{2}}}_g^2\,\dd\mu_g+\int_{M_k}\frac{\abs{\nabla\psi}^2_g}{\psi}\Rsc_+^p\,\dd\mu_g.
\end{align*} 
where we used $(1-\tfrac{4(p-1)}{p})\leq-\frac{1}{3}$ and applied Young's inequality in the form \eqref{eqn:Young2}. 
Consequently, for any $p\geq\frac{m}{2}$
\begin{align}
\frac{\partial}{\partial t}\int_{M_k}\psi\Rsc_+^p\,\dd\mu_g
&\leq(p-\tfrac{m}{2})\int_{M_k}\psi\Rsc_+^{p+1}\,\dd\mu_g
-(m-1)\int_{M_k}\abs[\big]{\nabla(\psi\Rsc_+^p)^{\frac{1}{2}}}^2\,\dd\mu_g
\notag\\&\hphantom{{}={}}\label{eqn:180504-1}
+(m-1)\int_{M_k}\frac{\abs{\nabla\psi}^2_g}{\psi}\Rsc_+^p\,\dd\mu_g. 
\end{align}
Let $\varphi_2\in C_c^{\infty}(B_{r_0+3})$ be a cut-off function such that $0\leq\varphi_2\leq1$, such that $\varphi_2|_{B_{r_0+2}}\equiv1$ and such that $\abs{\nabla \varphi_2}_{g_{M }}\leq 2$. 
Given $0<\alpha<1$, we set $\psi_2=\varphi_2^{\frac{2}{1-\alpha}}$. 
Then, 
\begin{align*}
\abs{\nabla\psi_2}_{g}^2
&=\frac{1}{u}\abs{\nabla\psi_2}_{g_{M }}^2
\leq\frac{1}{\ubelow}\Bigr(\frac{4}{1-\alpha}\Bigr)^2\varphi_2^{\frac{4}{1-\alpha}-2}
=\frac{1}{\ubelow}\Bigr(\frac{4}{1-\alpha}\Bigr)^2\psi_2^{1+\alpha}.
\end{align*}
With this choice of $\psi_2$ and any $\lambda>0$, H\"older's inequality and Young's inequality in the form \eqref{eqn:Young1} yield
\begin{align}\notag
&\int_{M_k}\frac{\abs{\nabla\psi_2}_g^2}{\psi_2}\Rsc_+^p\,\dd\mu_g
\\[1ex]\notag
&\leq \frac{1}{\ubelow}\Bigr(\frac{4}{1-\alpha}\Bigr)^2
\int_{M_k}\psi_2^{\alpha}\Rsc_+^p\,\dd\mu_g
\\[1ex]\notag
&\leq \frac{1}{\ubelow}\Bigr(\frac{4}{1-\alpha}\Bigr)^2\Bigl(\int_{M_k}(\psi_2\Rsc_+^p)^{\frac{m}{m-2}}\,\dd\mu_g\Bigr)^{\frac{m-2}{m}\alpha}
\Bigl(\int_{B_{r_0+3}}\Rsc_+^{p\frac{1-\alpha}{1-\frac{m-2}{m}\alpha}}\,\dd\mu_g\Bigr)^{1-\frac{m-2}{m}\alpha}
\\[1ex]\label{eqn:180504-2}
&\leq \alpha \lambda \Bigl(\int_{M_k}(\psi_2\Rsc_+^p)^{\frac{m}{m-2}}\,\dd\mu_g\Bigr)^{\frac{m-2}{m}}
\\&\hphantom{{}={}}\notag
+(1-\alpha)\lambda^{-\frac{\alpha}{1-\alpha}}
\ubelow^{-\frac{1}{1-\alpha}}
\Bigr(\frac{4}{1-\alpha}\Bigr)^{\frac{2}{1-\alpha}}\Bigl(\int_{B_{r_0+3}}\Rsc_+^{p\frac{1-\alpha}{1-\frac{m-2}{m}\alpha}}\,\dd\mu_g\Bigr)^{\frac{1-\frac{m-2}{m}\alpha}{1-\alpha}}. 
\end{align}
Restricting to the case $p=\frac{m}{2}$, we choose $\alpha=\frac{m}{m+8}$ such that $\frac{1-\alpha}{1-\frac{m-2}{m}\alpha}=\frac{4}{5}$. 
Then we choose $\lambda>0$ such that 
\begin{align*}
\alpha\lambda\leq\Bigl(\frac{\inf u}{\sup u}\Bigr)^{\frac{m-2}{2}}\Yamabe,
\end{align*}
where we again depend on the uniform upper and lower bound on $u$ from Lemma~\ref{lem:pointwise}. 
With these choices and the Sobolev estimate \eqref{eqn:Sobolev}, we obtain 
\begin{align*}
\frac{\partial}{\partial t}\int_{M_k}\psi_2\Rsc_+^{\frac{m}{2}}\,\dd\mu_g
&\leq C\Bigl(\int_{B_{r_0+3}}\Rsc_+^{\frac{2m}{5}}\,\dd\mu_g\Bigr)^{\frac{5}{4}}. 
\end{align*}
In particular, using \eqref{eqn:180502-3} from step 1, we conclude
\begin{align}
\int_{B_{r_0+2}}\Rsc_+^{\frac{m}{2}}\,\dd\mu_g\leq C_{m,r_0,T}
\end{align}
where the constant does not depend on $r$.

\emph{Step 3.} 
Let $\psi_3\in C_c^{\infty}(B_{r_0+2})$ be a cut-off function such that $\psi_3|_{B_{r_0+1}}\equiv 1$. 
This step is based on the estimate 
\begin{align*}
(p-\tfrac{m}{2})\int_{M_k}\psi_3\Rsc_+^{p+1}\,\dd\mu_g
&\leq (p-\tfrac{m}{2})
\Bigl(\int_{B_{r_0+2}}\Rsc_+^{\frac{m}{2}}\,\dd\mu_g\Bigr)^{\frac{2}{m}}
\Bigl(\int(\psi_3\Rsc_+^p)^{\frac{m}{m-2}}\,\dd\mu_g\Bigr)^{\frac{m-2}{m}}. 
\end{align*}
By step 2, $\nm{\Rsc_+}_{L^{\frac{m}{2}}(B_{r_0+2})}$ is bounded uniformly in $r$. 
If $p=\beta\frac{m}{2}$ with $\beta>1$ sufficiently close to $1$, then 
\begin{align*}
(p-\tfrac{m}{2})\nm{\Rsc_+}_{L^{\frac{m}{2}}(B_{r_0+2})}\leq \frac{m-1}{2}\Bigl(\frac{\inf u}{\sup u}\Bigr)^{\frac{m-2}{2}}\Yamabe
\end{align*}
and we can conclude as in step 2.

\emph{Step 4.} Let $\varphi_4\in C_c^{\infty}(B_{r_0+1})$ be a cut-off function such that $0\leq\varphi_4\leq1$, such that $\varphi_4|_{B_{r_0}}\equiv1$ and such that $\abs{\nabla \varphi_4}_{g_{M }}\leq 2$. 
As in step 2, we set $\psi_4=\varphi_4^{\frac{2}{1-\alpha}}$ with $1<\alpha<1$.

We apply Lemma \ref{lem:Yamabeinvariant} to estimate \eqref{eqn:180504-1} and obtain 
\begin{align}\notag
&\frac{\partial}{\partial t}\int_{M_k}\psi_4\Rsc_+^p\,\dd\mu_g
\\\notag
&\leq(p-\tfrac{m}{2}+\tfrac{m-2}{4})\int_{M_k}\psi_4\Rsc_+^{p+1}\,\dd\mu_g
-(m-1)\Yamabe\Bigl(\int_{M_k}(\psi_4\Rsc_+^p)^{\frac{m}{m-2}}\,\dd\mu_g\Bigr)^{\frac{m-2}{m}}
\\&\hphantom{{}={}} \label{eqn:180504-3}
+(m-1)\int_{M_k}\frac{\abs{\nabla\psi_4}^2_g}{\psi_4}\Rsc_+^p\,\dd\mu_g.
\end{align}
As in shown in \eqref{eqn:180504-2} we have 
\begin{align}\notag
&\int_{M_k}\frac{\abs{\nabla\psi_4}_g^2}{\psi_4}\Rsc_+^p\,\dd\mu_g
\leq \alpha \lambda \Bigl(\int_{M_k}(\psi_4\Rsc_+^p)^{\frac{m}{m-2}}\,\dd\mu_g\Bigr)^{\frac{m-2}{m}}
\\&\hphantom{{}={}} \label{eqn:180504-4}
+(1-\alpha)\lambda^{-\frac{\alpha}{1-\alpha}}
\ubelow^{-\frac{1}{1-\alpha}}
\Bigr(\frac{4}{1-\alpha}\Bigr)^{\frac{2}{1-\alpha}}\Bigl(\int_{B_{r_0+1}}\Rsc_+^{p\frac{1-\alpha}{1-\frac{m-2}{m}\alpha}}\,\dd\mu_g\Bigr)^{\frac{1-\frac{m-2}{m}\alpha}{1-\alpha}}. 
\end{align}
for any $\lambda>0$. 
This time we choose $0<\alpha<1$ depending on $m$ and $p$ such that 
\begin{align*} 
p\frac{1-\alpha}{1-\frac{m-2}{m}\alpha}=\frac{m}{2}.
\end{align*}
Then, we choose $\lambda=\frac{1}{2\alpha}\Yamabe>0$. 
Let $\beta>1$ as in step 3. 
H\"older's inequality and Young's inequality yield  
\begin{align}\notag
&\int_{M_k}\psi\Rsc_+^{p+1}\,\dd\mu_g
\\\notag  
&\leq\Bigl(\int_{M_k}(\psi\Rsc_+^{p})^{\frac{m}{m-2}}\,\dd\mu_g\Bigr)^{\frac{m-2}{m}\gamma}
\Bigl(\int_{M_k}\psi^{\frac{1-\gamma}{1-\frac{m-2}{m}\gamma}}\Rsc_+^{\beta\frac{m}{2}}\,\dd\mu_g\Bigr)^{1-\frac{m-2}{m}\gamma}
\\\label{eqn:180504-5} 
&\leq\gamma\delta\Bigl(\int_{M_k}(\psi\Rsc_+^{p})^{\frac{m}{m-2}}\,\dd\mu_g\Bigr)^{\frac{m-2}{m}}
+(1-\gamma)\delta^{-\frac{\gamma}{1-\gamma}}
\Bigl(\int_{B_{r_0+1}} \Rsc_+^{\beta\frac{m}{2}}\,\dd\mu_g\Bigr)^{\frac{1-\frac{m-2}{m}\gamma}{1-\gamma}}
\end{align}
where $\delta>0$ is arbitrary but $0<\gamma<1$ must satisfy $p+1=p\gamma+\beta(\frac{m}{2}-\frac{m-2}{2}\gamma)$. 
We solve the equation for
\begin{align*}
\gamma&=\frac{p-\beta\frac{m}{2}+1}{p-\beta\frac{m}{2}+\beta} 
\end{align*}
which indeed satisfies $0<\gamma<1$ since $\beta>1$ and compute 
\begin{align*}
\frac{1-\frac{m-2}{m}\gamma}{1-\gamma}
&= \frac{2(p-\frac{m}{2}+1)}{m(\beta-1)}. 
\end{align*}
Finally we chose $\delta>0$ such that 
$(p-\tfrac{m}{2}-\tfrac{m-2}{4})\gamma\delta<\tfrac{1}{2}(m-1)\Yamabe$ 
and combine \eqref{eqn:180504-3}, \eqref{eqn:180504-4} and \eqref{eqn:180504-5} to
\begin{align}\notag
&\frac{\partial}{\partial t}\int_{M_k}\psi_4\Rsc_+^p\,\dd\mu_g
\\
&\leq C_{m,p,\ubelow}\Bigl(\int_{B_{r_0+1}}\Rsc_+^{\frac{m}{2}}\,\dd\mu_g\Bigr)^{\frac{2p}{m}}
+ C_{m,p,r_0}\Bigl(\int_{B_{r_0+1}}\Rsc_+^{\beta\frac{m}{2}}\,\dd\mu_g\Bigr)^{\frac{2(p-\frac{m}{2}+1)}{m(\beta-1)}}
\end{align}
With the estimates from steps 2 and 3, the claim follows. 
\end{proof}

\begin{lem}\label{lem:180611}
Given $k>4$, let $u\in C^{2,\alpha;1,\frac{\alpha}{2}}(M_k\times[0,T])$ be a solution to problem~\eqref{eqn:pde} and let $\Rsc_{g(t)}$ be the scalar curvature of the Riemannian metric $g(t)=u(\cdot,t)g_{M}$ in $M_k$. 
Let $0<r_0<k-4$ be fixed and let $p>\frac{m}{2}$ be any exponent. 
Then, for every $t\in[0,T]$ 
\begin{align*}
\int_{B_{r_0}}\abs{\Rsc_{g(t)}}^p\,\dd\mu_{g(t)}\leq C
\end{align*}
where the constant $C$ increases with $p$ and $r_0$ but does not depend on $k$. 
\end{lem}

\begin{proof}
In view of Lemma \ref{lem:180507-1}, it remains to prove a similar estimate for the negative part $\Rsc_-(x,t)=\max\{0,-\Rsc_{g(t)}(x)\}$. 
Since $\Rsc=\Rsc_+-\Rsc_-$, we have 
\begin{align}\notag
&\frac{\partial}{\partial t}\int_{M_k}\psi\Rsc_-^p\,\dd\mu_g
\\[1ex]\notag
&=p\int_{M_k}\psi\Rsc_-^{p-1}\bigl(-\Rsc^2-(m-1)\Delta_g\Rsc\bigr)\,\dd\mu_g
-\frac{m}{2}\int_{M_k}\psi\Rsc_-^{p}\Rsc\,\dd\mu_g
\\[1ex]\notag
&=(-p+\tfrac{m}{2})\int_{M_k}\psi\Rsc_-^{p+1}\,\dd\mu_g
-p(m-1)\int_{M_k}\psi\Rsc_-^{p-1}\Delta_g\Rsc\,\dd\mu_g
\\[1ex]\notag
&=(-p+\tfrac{m}{2})\int_{M_k}\psi\Rsc_-^{p+1}\,\dd\mu_g
\\&\hphantom{{}={}}\notag
+p(m-1)\Bigl(\int_{M_k} \Rsc_-^{p-1}\sk{\nabla\psi,\nabla\Rsc}_g
+\psi\sk{\nabla\Rsc_-^{p-1},\nabla\Rsc}
\,\dd\mu_g\Bigr)
\\[1ex]\notag
&=(-p+\tfrac{m}{2})\int_{M_k}\psi\Rsc_-^{p+1}\,\dd\mu_g
\\&\hphantom{{}={}}\label{eqn:180615-1}
+(m-1)\Bigl(-\int_{M_k} \sk{\nabla\psi,\nabla\Rsc_-^p}_g\,\dd\mu_g
-\frac{4(p-1)}{p}\int_{M_k}\psi\abs{\nabla\Rsc_-^{\frac{p}{2}}}_g^2\,\dd\mu_g\Bigr).
\end{align}
We choose $\psi=\varphi^{2(p+1)}$, where $\varphi\in C_c^{\infty}(B_{r_0+4})$ satisfies $0\leq\varphi\leq1$, $\varphi|_{B_{r_0}}\equiv1$ and $\abs{\nabla \varphi}_{g_{M}}\leq 2$ as in step 1 of the proof of Lemma \ref{lem:180507-1} and estimate 
\eqref{eqn:180615-1} as in \eqref{eqn:180611-1}. Thus, 
\begin{align*}
&\frac{\partial}{\partial t}\int_{M_k}\psi\Rsc_-^p\,\dd\mu_g
\\[1ex]
&\leq\Bigl(\frac{(m-1)\lambda p}{p+1}-(p-\tfrac{m}{2})\Bigr)\int_{M_k}\psi\Rsc_-^{p+1}\,\dd\mu_g
+\frac{C_{m,p}}{\lambda^p}\Bigl(\int_{B_{r_0+4}}u^{\frac{m}{2}-p-1}\,\dd\mu_{g_{M }}\Bigr)
\end{align*}
where the parameter $\lambda>0$ is arbitrary. 
Since $p>\frac{m}{2}$ we may choose   
$\lambda=\frac{(p-\frac{m}{2})(p+1)}{(m-1)p}$ to obtain
\begin{align*}
\frac{\partial}{\partial t}\int_{M_k}\psi\Rsc_-^p\,\dd\mu_g
&\leq C_{m,p}
\Bigl(\int_{B_{r_0+4}}u^{\frac{m}{2}-p-1}\,\dd\mu_{g_{M }}\Bigr). 
\end{align*}
Since $u\geq\ubelow>0$ by Lemma \ref{lem:pointwise}, we conclude
\begin{align*}
\int_{B_{r_0}}\Rsc_-^{p}(\cdot,t)\,\dd\mu_{g(t)}
&\leq \int_{B_{r_0+4}}\abs{\Rsc_{g(0)}}^p\,\dd\mu_{g_0}
+t C_{m,r_0,p,\ubelow}.\qedhere
\end{align*}
\end{proof}

\begin{rem*}
We prove the $L^p(B_{r_0})$-estimates for $\Rsc_+$ and $\Rsc_-$ separately rather than directly estimating $\nm{\Rsc}_{L^p(B_{r_0})}$ because it is interesting to see that the estimate for the negative part of scalar curvature is simpler than the estimate for the positive part if $p>\frac{m}{2}$. 
The sign of the non-linear term $(p-\frac{m}{2})\Rsc_+^{p+1}$ respectively $(-p+\frac{m}{2})\Rsc_-^{p+1}$ makes the difference. 
\end{rem*}

\begin{proof}[Proof of Theorem \ref{thm:existence}]
Fix $T>0$. 
For every $k\in\mathbb{N}$, let $u_k\in C^{2,\alpha;1,\frac{\alpha}{2}}(M_k\times[0,T])$ denote the solution to problem \eqref{eqn:pde}. 
Let $U_k=u_k^{\eta}$ for $\eta=\frac{m-2}{4}$. 
Lemmata \ref{lem:pointwise} and~\ref{lem:180611} imply that for every fixed $t\in[0,T]$, the sequence $\{U_k(\cdot,t)\}_{4<k\in\mathbb{N}}$ is  bounded in the Sobolev space $W^{2,p}(M_1)$ for any $p>\frac{m}{2}$. 
In fact, we may apply the Calderon--Zygmund Inequality 
\cite[Theorem 9.11]{Gilbarg2001}
to the elliptic equation   
\begin{align*}
\Delta_{g_M} U_k&=\frac{\eta}{(m-1)}\bigl(\Rsc_{g_M} U_k-U_k^{1+\frac{1}{\eta}}\Rsc_{u_kg_M}\bigr)
\end{align*}
in $M_1\subset B_2$. Moreover, since 
\begin{align*}
\frac{\partial U_k}{\partial t}&=-\eta U_k\Rsc_{u_kg_M},
\end{align*}
we obtain that the sequence $\{U_k\}_{4\leq k\in\mathbb{N}}$ is bounded in $W^{1,p}(M_1\times[0,T])$ for any fixed $p>\frac{m}{2}$. 
If we choose $p=2(m+1)$, then Sobolev's embedding $W^{1,p}(M_1\times[0,T])\hookrightarrow C^{0,\alpha}(M_1\times[0,T])$ is compact for any $0<\alpha<\frac{1}{2}$ (recall that $M_1\subset M$ is smooth and bounded) and we obtain a subsequence $\Lambda_1\subset\mathbb{N}$ such that 
\begin{align*}
\{U_k|_{M_1\times[0,T]}\}_{4< k\in\Lambda_1}
\end{align*}
converges in $C^{0,\alpha}(M_1\times[0,T])$ to some $V_1$.
In particular, $\{U_k(\cdot,t)\}_{4< k\in\Lambda_1}$ converges to $V_1(\cdot,t)$ in $C^{0,\alpha}(M_1)$ for every fixed $t\in[0,T]$. 
As observed above, $\{U_k(\cdot,t)\}_{4< k\in\Lambda_1}$ is bounded in $W^{2,p}(M_1)$ which compactly embeds into $C^{1,\alpha}(M_1)$. 
Hence, a subsequence converges in $C^{1,\alpha}(M_1)$ and its limit must be $V_1(\cdot,t)$. 
Thus, $V_1\in C^{1,\alpha;0,\alpha}(M_1\times [0,T])$. 
Passing to the limit in the weak formulation of equation \eqref{eqn:180609}, we conclude that $V_1$ is a weak solution to equation \eqref{eqn:180609} in $M_1\times [0,T]$.
By parabolic regularity theory, $V_1$ is actually regular and a classical solution. 

We repeat this argument to obtain a subsequence $\Lambda_2\subset\Lambda_1$ such that  
\begin{align*}
\{U_k|_{M_2\times[0,T]}\}_{5<k\in\Lambda_2}
\end{align*}
converges in $C^{0,\alpha}(M_2\times[0,T])$ to some solution $V_2$ of equation \eqref{eqn:180609} in $M_2\times[0,T]$.
Iterating this procedure leads to a diagonal subsequence of $\{U_k\}_{4<k}$ which converges to a limit $U$ satisfying the Yamabe flow equation \eqref{eqn:180609} in $M\times[0,T]$.  
Moreover, the uniform bounds from Lemmata \ref{lem:pointwise} and \ref{lem:Rsc} are preserved in the limit and by construction, the initial condition is satisfied. 
\end{proof}


\clearpage 
\bibliography{Schulz-Yamabe_Flow}
 
\end{document}